\input ./style/arxiv-vmsta.cfg
\documentclass[numbers,compress,v1.0.1]{vmsta}
\usepackage{vtexbibtags}

\volume{7}
\issue{1}
\pubyear{2020}
\firstpage{61}
\lastpage{78}
\aid{VMSTA147}
\doi{10.15559/20-VMSTA147}
\articletype{research-article}

\startlocaldefs

\newtheorem{thm}{Theorem}
\newtheorem{lemma}{Lemma}
\newtheorem{cor}{Corollary}
\newtheorem{prop}{Proposition}

\theoremstyle{definition}
\newtheorem{definition}{Definition}
\newtheorem{rem}{Remark}
\newtheorem{example}{Example}

\allowdisplaybreaks
\ifdefined\HCode 
\def\index#1{}
\else 
\fi
\endlocaldefs

\begin{document}

\begin{frontmatter}
\pretitle{Research Article}
\title{The laws of iterated and triple logarithms for extreme values of regenerative processes}

\author{\inits{A.V.}\fnms{Alexander}~\snm{Marynych}\thanksref{cor1}\ead[label=e1]{marynych@unicyb.kiev.ua}}
\author{\inits{I.K.}\fnms{Ivan}~\snm{Matsak}\ead[label=e2]{ivanmatsak@univ.kiev.ua}}
\thankstext[type=corresp,id=cor1]{Corresponding author.}
\address{\institution{Taras Shevchenko National University of Kyiv},\break
Faculty of Computer Science and Cybernetics, 01601 Kyiv, \cny{Ukraine}}
%



\markboth{A.V. Marynych and I.K. Matsak}{Extreme values of regenerative processes}

\begin{abstract}
We analyze almost sure asymptotic behavior of extreme values of a regenerative
process. We show that under certain conditions a properly centered and
normalized running maximum of a regenerative process satisfies a law of
the iterated logarithm for the $\limsup $ and a law of the triple logarithm
for the $\liminf $. This complements a previously known result of Glasserman
and Kou [Ann. Appl. Probab. 5(2) (1995), 424--445]. We apply our results
to several queuing systems and a birth and death process.
\end{abstract}



\received{\sday{11} \smonth{1} \syear{2020}}
\revised{\sday{1} \smonth{2} \syear{2020}}
\accepted{\sday{1} \smonth{2} \syear{2020}}
\publishedonline{\sday{17} \smonth{2} \syear{2020}}

\end{frontmatter}

\section{Introduction and main results}
\label{sec1}

Various problems related to asymptotic behavior of extreme values of regenerative
processes\index{regenerative processes} is of considerable practical interest and has attracted a lot
of attention in probabilistic community. For example, extremes in queuing
systems\index{extremes in queuing systems} and of birth and death processes have been investigated in~\cite{ande,asm,coh,igl,serf},
to name but a few. Analysis carried out in the above papers is mostly based
on the classical theory of extreme values for independent identically distributed
(i.i.d.) random variables. A survey of early results in this direction
can be found, among other, in paper~\cite{asm}. In recent paper~\cite{ok_ik}
a slightly different approach to the asymptotic analysis of extreme values
of regenerative processes\index{regenerative processes} using a nonlinear time transformations has been
proposed.

The aforementioned works were mostly aimed at the derivation of \emph{weak}
limit theorems for extremes of regenerative processes.\index{regenerative processes extremes} In this article
instead, we are interested in almost sure (a.s.) behavior of general
regenerative processes\index{regenerative processes} and in particular of regenerative processes\index{regenerative processes} appearing
in queuing and birth--death systems. Our main results formulated in
Theorems~\ref{thm:main1} and \ref{thm:main2} below provide the laws of
iterated and triple logarithms for the running maximum of regenerative
processes.\index{regenerative processes} A distinguishing feature of our results is a different scaling
required for $\limsup $ and $\liminf $. Under the assumption that the right
tail of the maximum of a regenerative process\index{regenerative processes} over its regeneration cycle has
an exponential tail, this type of behavior has already been observed in
\cite{Glass}, see Proposition~3.2 therein. Our theorems provide a
generalization of the aforementioned result and cover, for example,
regenerative processes\index{regenerative processes} with Weibull-like tails of the maximum over a
regeneration cycle. As in many other papers dealing with extremes of
regenerative processes,\index{regenerative processes extremes} our approach relies on analyzing the a.s. behavior of
the running maximum of i.i.d. random variables. In this respect, let us also
mention papers \cite{Klass1984,Klass1985,Robbins+Siegmund:1972} dealing with
a.s.~growth rate of the running maximum, see Section 3.5 in
\cite{Embrechts_et_al:2013} for a survey.

Before formulating the results we introduce necessary definitions. Let us
recall, see \cite{BGT}, that a positive measurable function $U$ defined in
some neighbourhood of $+\infty $ is called regularly varying at $+\infty $
with index $\kappa \in \mathbb{R}$ if $U(x)=x^{\kappa } V(x)$, and the
function $V$ is slowly varying at $+\infty $, that is
\begin{equation*}
\lim _{t\to +\infty } \frac{V(tx)}{V(t)} = 1\quad \text{for all } x>0.
\end{equation*}

Given a function $H:\mathbb{R}\to \mathbb{R}$ we denote by $H^{-1}$ its
generalized inverse defined by
%
\begin{equation}\label{eq:gen_inverse_def}
H^{-1}(y)=\inf \left \{  x\in \mathbb{R}: H(x)>y\right \}  ,\quad y\in
\mathbb{R}.
\end{equation}

The following definition is of crucial importance for our main results.

\begin{definition}
We say that a function $H:\mathbb{R}\to \mathbb{R}$ satisfies condition
$(\mathbb{U})$ if the following holds:
\begin{enumerate}
\item $\lim _{x\to +\infty }H(x)=+\infty $;
\item the function $H$ is eventually nondecreasing and differentiable;
\item the derivative $h(x):=H'(x)$ is such that for some
$\kappa \in \mathbb{R}$ the function
\begin{equation*}
\hat{h}(x) = (H^{-1}(x))'=\frac{1}{h(H^{-1}(x))},\quad x\in \mathbb{R},
\end{equation*}
is regularly varying at $+\infty $ with index $\kappa $.
\end{enumerate}
\end{definition}

Note that the assumption of regular variation of $\hat{h}$ implies that $h$
is eventually positive. Thus, $H$ is eventually strictly increasing and the
generalized inverse $H^{-1}$ defined by \eqref{eq:gen_inverse_def} eventually
coincides with the usual inverse.

Let $X=(X(t))_{t\geq 0}$ be a regenerative random process,\index{regenerative random process} that is,
\begin{equation*}
X(t) = \xi _{k} (t- S_{k-1}) \quad \mbox{for } t \in [S_{k-1} , S_{k}), \
k\in \mathbb{N},
\end{equation*}
where
\begin{equation*}
S_{0}=0,\qquad S_{k} = T_{1} + \cdots + T_{k},\quad k \in \mathbb{N},
\end{equation*}
and $(T_{k} , \xi _{k} (\cdot ))_{k\in \mathbb{N}}$ is a sequence of
independent copies of a pair $(T , \xi (\cdot ))$, see, for example,
\cite[Part II, Chapter 2]{ks} and \cite[Chapter 11, \S 8]{fe}. The points
$(S_{k})$ are called regeneration epochs and the interval $[S_{k-1} , S_{k})$
is the $k$-th period of regeneration.\index{regeneration}

For $t\geq 0$, put
\begin{equation*}
\bar{X}(t) = \sup _{0\leq s < t} X(s),
\end{equation*}
and note that $\bar{X}(T_{1})$ is the maximum of the process $X$ on the first
period of regeneration.\index{regeneration} Let $F$ be the distribution function of
$\bar{X}(T_{1})$, that is,
\begin{equation*}
F(x):= \mathbf{P}(\bar{X}(T_{1})\leq x).
\end{equation*}
Put
\begin{equation*}
R(x):=-\log (1-F(x)),\quad x\in \mathbb{R},
\end{equation*}
and
\begin{equation*}
\alpha _{T} = \mathbf{E}T_{1} = \mathbf{E}T.
\end{equation*}
Note also that it is always possible to write a decomposition
%
\begin{equation}\label{f1}
R(x)= R_{0} (x)+ R_{1} (x),\quad x\in \mathbb{R},
\end{equation}
where
%
\begin{equation}\label{f150}
|R_{1} (x)| \leq C_{1} < \infty ,\quad x\in \mathbb{R}.
\end{equation}
Here and hereafter we denote by $C, C_{1} , C_{2} $ etc. some positive
constants which may vary from place to place and may depend on parameters of
the process $X(\cdot )$.

We are ready to formulate our first result.

\begin{thm}\label{thm:main1}
Let $(X(t))_{t\geq 0}$ be a regenerative random process.\index{regenerative random process} Assume that there
exists a decomposition \eqref{f1} such that \eqref{f150} holds and the
function $R_{0}$ satisfies condition $(\mathbb{U})$. Suppose further that
$\alpha _{T} < \infty $. For large enough $x\in \mathbb{R}$, let $r_{0}$ be
the derivative of $R_{0}$. Then
%
\begin{equation}\label{f2}
\limsup _{t\rightarrow \infty } \frac{ r_{0} (A_{0} (t)) (\bar{X}(t)
-A_{0}(t))}{L_{2} (t)}=1 \quad \textrm{a.s.},
\end{equation}
and
%
\begin{equation}\label{f3}
\liminf _{t\rightarrow \infty } \frac{ r_{0} (A_{0} (t)) (\bar{X}(t) -A_{0}
(t))}{L_{3} (t) }=-1 \quad \textrm{a.s.},
\end{equation}
where
\begin{equation*}
A_{0}{(t)}=R_{0}^{-1}\left (\log {\frac{t}{\alpha _{T}}}\right ), \qquad L_{2}
(t) = \log \log t,\qquad L_{3} (t) = \log \log \log t .
\end{equation*}
\end{thm}

Our next result is a counterpart of Theorem~\ref{thm:main1} for discrete
processes taking values in some lattice in $\mathbb{R}$. Such processes are
important, among other fields, in the queuing theory. Assume that
%
\begin{equation}\label{f151}
\mathbf{P}(X(t)\in \{0,1,2,3,\ldots \})=1,\quad t\geq 0,
\end{equation}
and, for $k=0,1,2,3,\ldots{} $, put
\begin{equation*}
R_{k}:=-\log \mathbf{P}(\bar{X}(T_{1}) > k).
\end{equation*}
Similarly to \eqref{f1} and \eqref{f150} we can write a decomposition
%
\begin{equation}\label{f152}
R_{k}= R_{0}(k) + R_{1}(k),\quad k=0,1,2,3,\ldots ,
\end{equation}
where $R_{0}:\mathbb{R}\to \mathbb{R}$ and $R_{1}:\mathbb{R}\to \mathbb{R}$
are real-valued functions and $R_{1}$ is such that
%
\begin{equation}\label{f153}
|R_{1}(k)| \leq C_{1} < \infty , \quad k=0,1,2,3,\ldots .
\end{equation}

\begin{thm}\label{thm:main2}
Let $(X(t))_{t\geq 0}$ be a regenerative random process\index{regenerative random process} such that
\eqref{f151} holds. Assume that there exists a decomposition \eqref{f152}
such that \eqref{f153} is fulfilled and the function $R_{0}$ satisfies
condition $(\mathbb{U})$. Suppose also that $\alpha _{T} < \infty $.
\begin{itemize}
\item[(i)] The asymptotic relation
%
\begin{equation}\label{f154}
r_{0} (R_{0}^{-1} (x))= o(\log x),\quad x\to \infty ,
\end{equation}
entails
%
\begin{equation}
\label{f200}
\limsup _{t\rightarrow \infty }
\frac{ r_{0} (A_{0} (t)) (\bar{X}(t) -A_{0}
(t))}{L_{2} (t) }=1 \quad \textrm{a.s.}
\end{equation}
\item[(ii)] The asymptotic relation
%
\begin{equation}\label{f155}
r_{0} (R_{0}^{-1} (x))= o(\log \log x),\quad x\to \infty ,
\end{equation}
entails
%
\begin{equation}\label{f300}
\liminf _{t\rightarrow \infty } \frac{ r_{0} (A_{0} (t)) (\bar{X}(t) -A_{0}
(t))}{L_{3} (t) }=-1 \quad \textrm{a.s.}
\end{equation}
The functions $A_{0}$ and $r_{0}$ were defined in Theorem~\ref{thm:main1}.
\end{itemize}
\end{thm}

\begin{rem}
In the discrete setting we assume that there exist extensions of the
sequences $(R_{0}(k))$ and $(R_{1}(k))$ to functions defined on the whole
real line with the extension of $R_{0}$ being smooth. While such an
assumption might look artificial, it is necessary for keeping the paper
homogeneous and allows us to work both in continuous and discrete settings
with the same class of functions $\mathbb{U}$.
\end{rem}

The article is organized as follows. In Section~\ref{sec:prelim} we collect
and prove some auxiliary results needed in the proofs of our main theorems.
They are given in Section~\ref{sec:proofs}. In Section~\ref{sec:appl} we
apply Theorems~\ref{thm:main1} and~\ref{thm:main2} to some queuing systems\index{queuing system}
and birth--death processes.

\section{Preliminaries}\label{sec:prelim}

Let us consider a sequence $(\xi _{k})_{k\in \mathbb{N}}$ of independent
copies of a random variable $\xi $ with the distribution function $F_{\xi
}(x)=\mathbf{P}(\xi \leq x) =: 1 - \exp (-R_{\xi }(x))$. Put
%
\begin{equation}\label{f4}
z_{n} = \max _{1\leq i \leq n} \xi _{i} .
\end{equation}

The following result was proved in \cite{akmi}, see Theorem 1 therein.

\begin{lemma}\label{l1}
Assume that the distribution of $\xi $ is such that $R_{\xi }$ satisfies
condition $(\mathbb{U})$. With $a(n)=R_{\xi }^{-1}(\log n)$ it holds
%
\begin{equation}\label{f5}
\limsup _{n\rightarrow \infty } \frac{r_{\xi }(a(n))(z_{n}-a(n))}{L_{2}
(n)}=1 \quad \textrm{a.s.},
\end{equation}
and
%
\begin{equation}\label{f6}
\liminf _{n\rightarrow \infty } \frac{r_{\xi }(a(n))(z_{n}-a(n))}{L_{3}
(n)}=-1 \quad \textrm{a.s.},
\end{equation}
where, for large enough $x\in \mathbb{R}$,
\begin{equation*}
r_{\xi }(x):=R^{\prime }_{\xi }(x)= \frac{F^{\prime }_{\xi }(x)}{1-F^{\prime
}_{\xi }(x)}.
\end{equation*}
\end{lemma}

The proof of Lemma~\ref{l1}, given in \cite{akmi}, consists of two steps.
Firstly, the claim is established for the standard exponential distribution
$\tau ^{e}$, that is, assuming $\mathbf{P}(\xi \leq x)=\mathbf{P}(\tau ^{e}
\leq x)= 1- \exp (-x)$. In the second step the claim is proved for an
arbitrary $R_{\xi }$ using regular variation and the representation
%
\begin{equation}\label{f7}
R_{\xi }^{-1}(\tau ^{e}) \overset{d}{=} \xi \quad \mbox{and} \quad R_{ \xi
}^{-1}(z_{n} ^{e}) \overset{d}{=} z_{n}.
\end{equation}
Here and hereafter $z_{n} ^{e} = \max _{1\leq i \leq n} \tau ^{e}_{i}$ and
$(\tau ^{e}_{i})_{i\in \mathbb{N}}$ are independent copies of $\tau ^{e}$.

We need the following generalization of Lemma~\ref{l1}.

\begin{lemma}\label{l3}
Assume that the law of $\xi $ is such that the function $R_{\xi }$ possesses
a decomposition \eqref{f1} with $R_{1}$ satisfying \eqref{f150} and $R_{0}$
satisfying condition $(\mathbb{U})$. Then
%
\begin{equation}\label{f8}
\limsup _{n\rightarrow \infty } \frac{r_{0}(a_{0}(n))(z_{n}-a_{0}(n))}{L_{2}
(n)}=1 \quad \textrm{a.s.},
\end{equation}
and
%
\begin{equation}\label{f9}
\liminf _{n\rightarrow \infty } \frac{r_{0}(a_{0}(n))(z_{n}-a_{0}(n))}{L_{3}
(n)}=-1 \quad \textrm{a.s.},
\end{equation}
where $a_{0}{(n)}=R_{0}^{-1}(\log {n})$ and $r_{0} (x)=R_{0}^{\prime }(x)$.
\end{lemma}

To prove Lemma~\ref{l3} we need the following simple result, see Theorem~3.1
in \cite{bu}.

\begin{lemma}\label{l2}
Let $H$ be a function regularly varying at $+\infty $ and let $(c_{n})_{n\in
\mathbb{N}}$ and $(d_{n})_{n\in \mathbb{N}}$ be two sequences of real numbers
such that $\lim _{n\to \infty }c_{n}=+\infty $, $\lim _{n\to \infty
}c_{n}/d_{n}= 1$. Then
\begin{equation*}
\lim _{n\to \infty }\frac{H(c_{n})}{H(d_{n})}=1.
\end{equation*}
\end{lemma}

\begin{proof}[Proof of Lemma~\ref{l3}]
Fix a sequence of standard exponential random variables $(\tau
_{i}^{e})_{i\in \mathbb{N}}$ and assume without loss of generality that the
sequence $(z_{n})_{n\in \mathbb{N}}$ is constructed from $(\tau
_{i}^{e})_{i\in \mathbb{N}}$ via formula \eqref{f7}. The subsequent proof is
divided into two steps.

\noindent%
\textsc{Step 1.} Suppose additionally that the function $R_{0}$ is everywhere
nondecreasing, differentiable, and $R_{0} (-\infty )=0$. Then $F_{0} (x):=1 -
\exp (-R_{0}(x))$ is a distribution function. Put $\xi ^{\prime }_{i}=
R_{0}^{-1}(\tau ^{e}_{i})$ for $i\in \mathbb{N}$ and let $z_{n}^{\prime } =
\max _{1\leq i \leq n} \xi _{i}^{\prime }$. From Lemma~\ref{l1} we infer
%
\begin{equation}\label{f10}
\limsup _{n\rightarrow \infty } \frac{r_{0}(a_{0}(n))(z_{n}^{\prime
}-a_{0}(n))}{L_{2} (n)}=1 \quad \textrm{a.s.}
\end{equation}
Let $C_{1}$ be a constant such that \eqref{f150} holds. From the definition
of the function $R_{\xi }^{-1}$ and decomposition \eqref{f1} we obtain
\begin{equation*}
R_{0}^{-1}(x- C_{1})\leq R_{\xi }^{-1}(x) \leq R_{0}^{-1}(x+ C_{1}), \quad
x\in \mathbb{R},
\end{equation*}
and thereupon
\begin{equation*}
R_{0}^{-1}(z_{n}^{e}- C_{1})\leq R_{\xi }^{-1}(z_{n}^{e}) \leq
R_{0}^{-1}(z_{n}^{e}+ C_{1}).
\end{equation*}
Hence, by monotonicity of $R_{0}^{-1}$, we have
%
\begin{equation}\label{f11}
| R_{\xi }^{-1}(z_{n}^{e}) -R_{0}^{-1}(z_{n}^{e})|\leq R_{0}^{-1}(z_{n}^{e}+
C_{1}) - R_{0}^{-1}(z_{n}^{e}- C_{1}) = 2C_{1} \hat{r}_{0} (z_{n}^{e}+ C_{1}
(2\theta _{n} -1)),
\end{equation}
where the equality follows from the mean value theorem for differentiable
functions, $\hat{r}_{0}(x) = (R_{0}^{-1}(x))'$ and $ 0\leq \theta _{n} \leq
1$.

It is known, see \cite[Chapter 4, Example 4.3.3]{gal}, that
\begin{equation*}
\lim _{n\to \infty }\frac{z_{n}^{e}}{\log n}=1\quad \textrm{a.s.}
\end{equation*}
Thus, from Lemma~\ref{l2} we deduce
\begin{equation*}
\lim _{n\to \infty } \frac{\hat{r}_{0}(z_{n}^{e} + C_{1} (2\theta _{n}
-1))}{\hat{r}_{0} (\log n)}= 1 \quad \textrm{a.s.}
\end{equation*}
In conjunction with \eqref{f11} this yields
%
\begin{equation}\label{f12}
| R_{\xi }^{-1}(z_{n}^{e}) -R_{0}^{-1}(z_{n}^{e})|\leq 2C_{1} \hat{r}_{0}
(\log n) (1+o(1)) = \frac{2C_{1}}{r_{0} (a_{0} (n) )} (1+o(1)) .
\end{equation}
Taking together relations \eqref{f10}, \eqref{f12} we arrive at
\eqref{f8}.

Similarly, from Lemma~\ref{l1} we have
%
\begin{equation}\label{f13}
\liminf _{n\rightarrow \infty } \frac{r_{0}(a_{0}(n))(z_{n}^{\prime
}-a_{0}(n))}{L_{3} (n)}=-1 \quad \textrm{a.s.}
\end{equation}
Therefore, \eqref{f9} follows from \eqref{f13} and \eqref{f12}.

\noindent%
\textsc{Step 2.} Let us now turn to the general case where the function
$R_{0}$ is nondecreasing and differentiable on some interval $[x_{0} , \infty
)$ with $x_{0}>0$. Recall decomposition \eqref{f1}. Let
$\tilde{R}_{0}:\mathbb{R}\to \mathbb{R}$ and $\tilde{R}_{1}:\mathbb{R}\to
\mathbb{R}$ be arbitrary nondecreasing differentiable functions such that
\begin{align*}
\tilde{R}_{0}(x)&=R_{0}(x)\quad \text{and}\quad \tilde{R}_{1}(x)=R_{1}(x)
\quad \text{for } x\geq x_{0},
\\*
\tilde{R}_{0}(x)&=\tilde{R}_{1}(x)=0\quad \text{for } x\leq 0.
\end{align*}
Put
\begin{equation*}
\tilde{R}(x):=\tilde{R}_{0}(x)+\tilde{R}_{1}(x),\quad x\in \mathbb{R}.
\end{equation*}
The functions $\tilde{R}_{0}$, $\tilde{R}_{1}$ and $\tilde{R}$ satisfy all
the assumptions of Step 1. Thus, if we set
\begin{equation*}
\tilde{\xi }_{i} = \tilde{R}^{-1}(\tau _{i}^{e}) , \qquad \tilde{z}_{n} = \max
_{1\leq i \leq n} \tilde{\xi }_{i},
\end{equation*}
then the sequence $(\tilde{z}_{n})_{n\in \mathbb{N}}$ satisfies \eqref{f8}
and \eqref{f9} with the same normalizing functions $r_{0} (a_{0}(n))$ and
$a_{0}(n)$. The latter holds true since for sufficiently large $x>0$ we have
$\tilde{R}_{0}^{-1}(x) = {R}_{0}^{-1} (x)$.

It remains to note that the asymptotics of $(\tilde{z}_{n})$ and $({z}_{n})$
are the same. Indeed, set
\begin{equation*}
n_{0}: = \min (i\geq 1: \tau _{i}^{e} \geq y_{0} ),
\end{equation*}
where $y_{0} := R(x_{0}) = \tilde{R}(x_{0})$. Then ${z}_{n} = \tilde{z}_{n}$
for $n\geq n_{0}$ and we see that both \eqref{f8} and \eqref{f9} hold for
$({z}_{n})$ as well. This finishes the proof of Lemma~\ref{l2}.
\end{proof}

The next lemma is a counterpart of Lemma~\ref{l3} for discrete distributions.
Assume that $\xi $ has distribution
\begin{equation*}
\mathbf{P}(\xi =k ) = p_{k},
\end{equation*}
where $p_{k}\geq 0$ and $\sum _{k=0}^{\infty } p_{k} =1$. Put
\begin{equation*}
q(k) = \sum _{i> k} p_{i}=\exp (-R_{\xi ,k}).
\end{equation*}

\begin{lemma}\label{l4}
Let $\xi $ be a random variable taking values in $\{0,1,2,3,\ldots \}$ and
let $R_{\xi ,k}$ be such that there exists a decomposition \eqref{f152} with
$R_{1}$ satisfying \eqref{f153} and $R_{0}$ satisfying condition
$(\mathbb{U})$.
\begin{itemize}
\item[(i)] if \eqref{f154} holds, then $(z_{n})$ satisfies equality
    \eqref{f8};
\item[(ii)] if \eqref{f155} holds, then $(z_{n})$ also satisfies
    \eqref{f9}.
\end{itemize}
\end{lemma}

\begin{proof}
Similarly to Lemma~\ref{l3} the proof is divided into two steps. We provide
the details only for the first step leaving the second step for an interested
reader. Thus, we put $\xi _{i}^{\prime } = R_{0}^{-1}(\tau ^{e}_{i})$ for
$i\in \mathbb{N}$. Note that $\xi _{i}^{\prime }$ are i.i.d.~with the
distribution function $F_{0} (x)=1 - \exp (-R_{0}(x))$. Thus, for
$z_{n}^{\prime } = \max _{1\leq i \leq n}\xi _{i}^{\prime }$, equality
\eqref{f10} holds.

Let us consider
\begin{equation*}
\lceil R_{\xi ,k}^{-1}(y) \rceil =\inf \left \{  k=0,1,2,3,\ldots : R_{ \xi
,k}\geq y\right \}  ,
\end{equation*}
then
\begin{equation*}
|R_{\xi ,k}^{-1}(y) - \lceil R_{\xi ,k}^{-1}(y) \rceil | \leq 1,
\end{equation*}
for all $y\in \mathbb{R}$, and therefore
\begin{equation*}
|R_{\xi ,k}^{-1}(z_{n}^{e}) - \lceil R_{\xi ,k}^{-1}(z_{n}^{e}) \rceil | \leq
1 , \qquad |R_{0}^{-1}(z_{n}^{e}) - \lceil R_{0}^{-1}(z_{n}^{e}) \rceil |
\leq 1 .
\end{equation*}
Further, condition \eqref{f153} and monotonicity of the function $\lceil
\cdot \rceil $ both imply
\begin{equation*}
\lceil R_{0}^{-1}(z_{n}^{e} -C_{1})\rceil \leq \lceil R_{\xi ,k}^{-1}(z_{n}^{e})
\rceil \leq \lceil R_{0}^{-1}(z_{n}^{e} +C_{1})\rceil .
\end{equation*}
Combining the above estimates, we derive
\begin{equation*}
R_{0}^{-1}(z_{n}^{e} -C_{1}) -2 \leq R_{\xi ,k}^{-1}(z_{n}^{e}) \leq R_{0}^{-1}(z_{n}^{e}
+C_{1}) +2.
\end{equation*}
This means
%
\begin{eqnarray}\label{f111}
| R_{\xi ,k}^{-1}(z_{n}^{e}) -R_{0}^{-1}(z_{n}^{e})|&\leq &
R_{0}^{-1}(z_{n}^{e}+ C_{1}) -R_{0}^{-1}(z_{n}^{e}- C_{1}) +4 \nonumber
\\
&\leq & \frac{2C_{1}}{r_{0} (a_{0} (n) )} (1+o(1)) +4,
\end{eqnarray}
see estimates (\ref{f11}), (\ref{f12}).

Assuming \eqref{f154} we see that \eqref{f8} holds. Similarly, condition
\eqref{f155} yields (\ref{f9}).
\end{proof}

The next simple lemma is probably known, however we prefer to give an
elementary few lines proof.

\begin{lemma}\label{l5}
For arbitrary $p>1$ and $b\in \mathbb{R}$ it holds
%
\begin{equation}\label{f45}
\Lambda _{n} := \sum _{k=1}^{n} \frac{p^{k}}{k^{b}} =
\frac{p^{n+1}}{(p-1)n^{b}} (1+o(1)),\quad n\to \infty .
\end{equation}
\end{lemma}

\begin{proof}
By the Stolz--Ces\'{a}ro theorem we have
\begin{align*}
\lim _{n\to \infty }\frac{(p-1)n^{b}\Lambda _{n}}{p^{n+1}}&=\lim _{n \to
\infty } \frac{\Lambda _{n}-\Lambda
_{n-1}}{\frac{p^{n+1}}{(p-1)n^{b}}-\frac{p^{n}}{(p-1)(n-1)^{b}}}
\\
&=\lim _{n\to \infty }
\frac{\frac{p^{n}}{n^{b}}}{\frac{p^{n+1}}{(p-1)n^{b}}-\frac{p^{n}}{(p-1)(n-1)^{b}}}=
\lim _{n\to \infty }\frac{p-1}{p-\frac{n^{b}}{(n-1)^{b}}}=1.
\end{align*}
The proof is complete.
\end{proof}

\section{Proofs of Theorems \ref{thm:main1} and \ref{thm:main2}}\label{sec:proofs}

\begin{proof}[Proof of Theorem~\ref{thm:main1}]
Let us start with a proof of equality \eqref{f2}. To this end, we introduce
the following notation
\begin{equation*}
Y_{k} = \sup _{ S_{k-1}\leq t < S_{k}} X(t) , \qquad Z_{n} = \max _{1 \leq k
\leq n} Y_{k},\quad k\in \mathbb{N}.\vadjust{\eject}
\end{equation*}
Since $(S_{k})$ are the moments of regeneration\index{regeneration} of the process $(X(t))_{t\geq
0}$, $(Y_{k})$ are i.i.d. random variables. Furthermore, it is clear that the
sequence $(Y_{k})$ satisfies conditions of Lemma~\ref{l3}. Therefore,
%
\begin{equation}\label{f14}
\limsup _{n\rightarrow \infty } \frac{r_{0}(a_{0}(n))(Z_{n}-a_{0}(n))}{L_{2}
(n)}=1 \quad \textrm{a.s.}
\end{equation}
Denote by $N$ the counting process for the sequence $(S_{k})$, that is,
\begin{equation*}
N(t) = \max \{k\geq 0 : S_{k} \leq t\}, \quad t\geq 0.
\end{equation*}
Since $\lim _{t\to \infty }N(t)=+\infty $ a.s. and $N(t)$ runs through all
positive integers, from \eqref{f14} we obtain
%
\begin{equation}\label{f15}
\limsup _{t\rightarrow \infty } \frac{r_{0}(R_{0}^{-1}(\log
N(t)))(Z_{N(t)}-R_{0}^{-1}(\log N(t)))}{L_{2} (N(t))}=1 \quad \textrm{a.s.}
\end{equation}
By the strong law of large numbers for $N$ we have
%
\begin{equation}\label{f160}
\lim _{t\rightarrow \infty }\frac{N(t)}{t} = \frac{1}{\alpha _{T} } \quad
\textrm{a.s.},
\end{equation}
whence, as $t\to \infty $,
\begin{equation*}
\log {N(t)} = \log \frac{t}{\alpha _{T} } + o(1) \quad \textrm{a.s.}
\end{equation*}
In what follows $o(1)$ is a random function which converges to zero a.s.~as
$t\to \infty $. Plugging the above representations into \eqref{f15}, we
obtain
%
\begin{equation}\label{f16}
\limsup _{t\rightarrow \infty } \frac{r_{0}(R_{0}^{-1}(\log \frac{t}{\alpha
_{T} } + o(1)))(Z_{N(t)}-R_{0}^{-1}(\log \frac{t}{\alpha _{T} } +
o(1)))}{L_{2} (\frac{t}{\alpha _{T} }(1 + o(1)) )}=1 \quad \textrm{a.s.}
\end{equation}
Further, by the slow variation of $L_{2}$ we can replace the denominator in
\eqref{f16} by $L_{2} (t)$.

Let us recall that under the assumptions of Theorem~\ref{thm:main1}, the
function $\hat{r}_{0}(x) = (R_{0}^{-1}(x))'$ is regularly varying at
infinity. So using once again Lemma~\ref{l2}, we obtain the equalities
\begin{equation*}
\hat{r}_{0}\biggl(\log \frac{t}{\alpha _{T} } + o(1)\biggr)= \hat{r}_{0}\biggl(\log
\frac{t}{\alpha _{T} } \biggr)(1+ o(1))= \frac{1 + o(1)}{r_{0}(R_{0}^{-1}(\log
\frac{t}{\alpha _{T} }))},\quad t\to \infty ,
\end{equation*}
and
\begin{equation*}
R_{0}^{-1}\biggl(\log \frac{t}{\alpha _{T} } + o(1)\biggr) -R_{0}^{-1}\biggl(\log
\frac{t}{\alpha _{T} }\biggr)= o(1) \hat{r}_{0}\biggl(\log \frac{t}{\alpha _{T} } +
o(1)\biggr),\quad t\to \infty .
\end{equation*}
Combining the above relations, from \eqref{f16} we derive
%
\begin{equation}\label{f17}
\limsup _{t\rightarrow \infty } \frac{r_{0}(A_{0}(t))(Z_{N(t)}-A_{0}(t)
)}{L_{2} (t)}=1 \quad \textrm{a.s.},
\end{equation}
with $A_{0}(t)$ and $r_{0}(t)$ as in Theorem~\ref{thm:main1}. The same
argument with $N(t)$ replaced by $N(t) + 1$ yields
\begin{equation*}
\limsup _{t\rightarrow \infty } \frac{r_{0}(A_{0}(t))(Z_{N(t)+1}-A_{0}(t)
)}{L_{2} (t)}=1 \quad \textrm{a.s.}
\end{equation*}
It remains to note that
\begin{equation*}
Z_{N(t)} \leq \bar{X}(t) \leq Z_{N(t) + 1} \quad \textrm{a.s.}
\end{equation*}
Summarizing this gives equality \eqref{f2}. The proof of the relation
\eqref{f3} utilizes equality \eqref{f9} of Lemma~\ref{l3} and is similar. We
omit the details.
\end{proof}

The derivation of Theorem~\ref{thm:main2} is based on Lemma~\ref{l4} and
basically repeats the proof of Theorem~\ref{thm:main1}. We leave the details
to a reader.

Suppose that under the assumptions of Theorem~\ref{thm:main1} the parameter
$t$ runs over a countable set $t\in \mathcal{T}:= \{t_{0} =0 < t_{1} < t_{2}
< \cdots \}$ such that $\lim _{i\to \infty }t_{i} =+\infty $ as $i \to \infty
$. The set $\mathcal{T}$ can be random and the process $X$ can depend on
$\mathcal{T}$. Assume that $\mathbb{P}(S_{i} \in \mathcal{T})=1 $ for all
$i\in \mathbb{N}$.

Put $X_{i}:=X(t_{i} )$ and $\bar{X}_{n} = \max _{0 \leq i \leq n}X_{i}$.
Assume that extreme values of the process $X$ are attained at the points of
the set $\mathcal{T}$. More precisely, for all $t\geq 0$,
%
\begin{equation}\label{f18}
\sup _{0 \leq s \leq t}X(s) = \max _{0 \leq t_{i} \leq t}X_{i} \quad
\textrm{a.s.}
\end{equation}

In what follows the next proposition will be useful.

\begin{prop}\label{p1}
Under the assumptions of Theorem~\ref{thm:main1} suppose that there exists a
set $\mathcal{T}$ such that condition \eqref{f18} holds and, further,
%
\begin{equation}\label{f19}
\lim _{n\rightarrow \infty } \frac{t_{n}}{n}=\alpha \quad \textrm{a.s.}
\end{equation}
Then
%
\begin{equation}\label{f20}
\limsup _{t\rightarrow \infty } \frac{ r_{0} (A (n)) (\bar{X}_{n} -A
(n))}{L_{2} (n) }=1 \quad \textrm{a.s.},
\end{equation}
and
%
\begin{equation}\label{f21}
\liminf _{n\rightarrow \infty } \frac{ r_{0} (A (n)) (\bar{X}_{n} -A
(n))}{L_{3} (n) }=-1 \quad \textrm{a.s.},
\end{equation}
where
\begin{equation*}
A{(n)}=R_{0}^{-1}\left (\log {\frac{\alpha n}{\alpha _{T}}}\right ).
\end{equation*}
\end{prop}

\begin{proof}
A proof given below is similar to the proof of Theorem~\ref{thm:main1}. From
equations \eqref{f160} and \eqref{f19} we obtain, as $n\to \infty $,
%
\begin{equation}\label{f22}
\frac{N(t_{n})}{n}= \frac{N(t_{n})}{t_{n}} \frac{t_{n}}{n } \rightarrow
\frac{\alpha }{\alpha _{T} } \quad \textrm{a.s.}
\end{equation}
Further, replacing $n$ by $N(t_{n})$ in equality \eqref{f14}, which is
possible because $N(t_{n})$ diverges to infinity through all positive
integers, we get
%
\begin{equation}\label{f23}
\limsup _{n\rightarrow \infty } \frac{r_{0}(R_{0}^{-1}(\log
N(t_{n})))(Z_{N(t_{n})}-R_{0}^{-1}(\log N(t_{n})))}{L_{2} (N(t_{n}))}=1 \quad
\textrm{a.s.}
\end{equation}
Directly from \eqref{f18} we derive
\begin{equation*}
Z_{N(t_{n})} \leq \bar{X}_{n} \leq Z_{N(t_{n}) + 1} \quad \textrm{a.s.}
\end{equation*}
These inequalities and relations \eqref{f22}, \eqref{f23} yield equality
\eqref{f20}. The same argument can be applied for proving \eqref{f21}.
\end{proof}

\section{Applications}\label{sec:appl}

\begin{example}[Queuing system $GI/G/1$]\label{example1}
Let us consider a single-channel queuing system with customers arriving at
$0=t_{0}<t_{1}<t_{2}<\cdots < t_{i}<\cdots $. Let $0=W_{0},W_{1},W_{2},
\ldots , W_{i}, \ldots $ be the actual waiting times of the customers. Thus,
at time $t=0$ a first customer arrives and the service starts immediately.
Denote by $\zeta _{i} = t_{i} - t_{i-1}$, for $i\in \mathbb{N}$, the
interarrival times between successive customers, and $\eta _{i}$, $i\in
\mathbb{N}$, is the service time of the $i$-th customer. Suppose that $(\zeta
_{i} )$ and $(\eta _{i})$ are independent sequences of \xch{i.i.d.}{i.i.d}
random variables. In the standard notation, this queuing system\index{queuing system} has the type
$GI/G/1$, see \cite{gk1,kar1}.

Let $W(t)$ be the waiting time of the last customer in the queue at time
$t\geq 0$, that is,
\begin{equation*}
W(t) ={W}_{\nu (t)},\quad \text{where }\nu (t)= \max (k\geq 0: t_{k} \leq t),
\end{equation*}
and
\begin{equation*}
W(t_{n}) ={W}(t_{n} + ) = W_{n}.
\end{equation*}
Set
\begin{equation*}
\bar{W}(t) = \sup _{0 \leq s \leq t}W(s) = \max _{0 \leq t_{k} \leq t} W_{k}
,
\end{equation*}
then
\begin{equation*}
\bar{W}_{n} = \max _{1\leq i \leq n}W_{i} = \bar{W}(t_{n}).
\end{equation*}
Denote $\mathbf{E}\zeta _{i} =a$, $\mathbf{E}\eta _{i} =b$ and assume that
both expectations are finite. Further, we impose the following conditions on
$\zeta _{i}$ and $\eta _{i}$:
%
\begin{equation}\label{f24}
\rho : = \frac{b}{a} <1
\end{equation}
and for some $\gamma > 0 $, it holds
%
\begin{equation}\label{f25}
\mathbf{E}\exp (\gamma (\eta _{i} - \zeta _{i})) =1, \qquad \mathbf{E}( \eta
_{i} - \zeta _{i})\exp (\gamma (\eta _{i} - \zeta _{i}))< \infty .
\end{equation}
Under these assumptions the evolution of the queuing system\index{queuing system} can be decomposed
into busy periods, when a customer is in service, interleaved by idle
periods, when the system is empty. Let us introduce regeneration moments
$(S_{k})$ of the process $W$ as follows: $S_{0} =0$ and, for $i\in
\mathbb{N}$, $S_{i}$ is the arrival time of a new customer at the end of
$i$-th idle period. Let $T_{i} $ be the duration of the $i$-th regeneration
cycle, and $\bar{W}(T_{1})$ be the maximum waiting time during the first
regeneration cycle. It is known, see \cite{asm} and \cite{igl}, that under
conditions \eqref{f24} and \eqref{f25}, we have
\begin{equation*}
\mathbf{P}(\bar{W}(T_{1}) > x) = (C+o(1)) \exp (-\gamma x) , \quad x
\rightarrow \infty .
\end{equation*}
Condition \eqref{f24} also guarantees that the average duration of the $i$-th
regeneration cycle is finite, that is, $\alpha _{T} = \mathbf{E}T_{i} <
\infty $, see \cite[Chapter 14, \S 3, Theorem 3.2]{kar1}.

Thus, if we set $X(t)={W}(t)$, $R_{0} (x)= \gamma x$, $R_{1} (x)= \log C +
o(1)$ and $r_{0} (x)= \gamma $, then from Theorem~\ref{thm:main1} and
Proposition \ref{p1} we derive the corollary.

\begin{cor}\label{n1}
Assume that the queuing system\index{queuing system} $GI/G/1$ satisfies conditions \eqref{f24} and
\eqref{f25}. Then
%
\begin{equation}\label{f26}
\limsup _{t\rightarrow \infty } \frac{ \gamma \bar{W}(t) -\log t}{L_{2} (t)
}= \limsup _{n\rightarrow \infty } \frac{\gamma \bar{W}_{n} -\log n}{L_{2}
(n) }=1 \quad \textrm{a.s.},
\end{equation}
and
%
\begin{equation}\label{f27}
\liminf _{t\rightarrow \infty } \frac{ \gamma \bar{W}(t) -\log t}{L_{3} (t)
}= \liminf _{n\rightarrow \infty } \frac{ \gamma \bar{W}_{n} -\log n}{L_{3}
(n) }=-1 \quad \textrm{a.s.}
\end{equation}
\end{cor}

\begin{rem}\label{z1}
\begin{itemize}
\item[(i)] Suppose that
\begin{equation*}
\mathbf{P}(\zeta _{i} \leq x) = 1 - \exp (-\lambda x),\qquad
\mathbf{P}(\eta _{i} \leq x) = 1 - \exp (-\mu x), \quad x \geq 0,
\end{equation*}
that is, we consider the queuing system\index{queuing system} $M/M/1$. Assume further, that $\rho
:= \lambda /\mu <1$. It is easy to check that conditions \eqref{f24} and
\eqref{f25} are fulfilled, and therefore equalities \eqref{f26} and
\eqref{f27} hold with $ \gamma = \mu - \lambda = \mu (1-\rho )$.
\item[(ii)] Suppose that
\begin{equation*}
\mathbf{P}(\zeta _{i} \leq x) = 1 - \exp (-\lambda x) , \quad x\geq 0,
\end{equation*}
and $\mathbf{P}(\eta _{i} = const = d)=1$. Assume further, that $\rho : =
\lambda d < 1$. Then relations \eqref{f24}--\eqref{f27} hold with $\gamma =
x_{\rho }/d$, with $x_{\rho }>0 $ being the unique positive root of the
equation
\begin{equation*}
e^{x} = 1 + \frac{x}{\rho }.
\end{equation*}
\end{itemize}
\end{rem}
\end{example}

\begin{example}[Queuing system $M/M/m$]\label{example2}
Let us now consider a queuing system\index{queuing system} with $m$ servers and customers which
arrive according to the Poisson process with intensity $\lambda $, and
service times being independent copies of a random variable $\eta $ with
an exponential distribution
\begin{equation*}
\mathbf{P}(\eta \leq x) = 1 - \exp (-\mu x),\quad x\geq 0.
\end{equation*}
In the standard notation, this queuing system\index{queuing system} has the type $M/M/m$, see
\cite{gk1,kar1}.

We impose the following assumption on the parameters $\lambda $ and $\mu $
ensuring existence of the stationary regime:
%
\begin{equation}\label{f30}
\rho : = \frac{\lambda }{m \mu } <1 .
\end{equation}
For $t\geq 0$, let $Q(t)$ denote the length of the queue at time $t$, that
is, the total number of customers in service or pending. Set
\begin{equation*}
\bar{Q}(t) = \sup _{0\leq s < t} Q(s),\quad t\geq 0.
\end{equation*}
In the same way as in Example~\ref{example1}, one can introduce regeneration
moments $(S_{k})$ for the process $Q$: $S_{0}:=0$ and, for $i\in \mathbb{N}$,
$S_{i}$ is the arrival time of a new customer after the $i$-th busy period.
Let $T_{i} $ be the duration of the $i$-th regeneration cycle and
$\bar{Q}(T_{1})$ be the maximum length of the queue in the first regeneration
cycle. Put
%
\begin{equation}\label{f31}
\mathbf{P}(\bar{Q}(T_{1}) >x) = \exp (-R(x)).
\end{equation}
In recent paper \cite{do_m} the authors established that function $R$ in
\eqref{f31} satisfies conditions \eqref{f152} and \eqref{f153} with
\begin{equation*}
R_{0} (x)= -x \log \rho , \qquad r_{0} (x)= -\log \rho .
\end{equation*}
Using Theorem~\ref{thm:main2} we infer

\begin{cor}\label{n2}
Assume that for a queuing system\index{queuing system} ${M/M/m}$, $1 \leq m < \infty $, the
condition \eqref{f30} is fulfilled. Then
%
\begin{equation}\label{f28}
\limsup _{t\rightarrow \infty } \frac{\bar{Q}(t)\log \frac{1}{\rho }-\log
{t}}{L_{2} (t) }=1\quad \textrm{a.s.},
\end{equation}
and
%
\begin{equation}\label{f29}
\liminf _{t\rightarrow \infty } \frac{\bar{Q}(t)\log \frac{1}{\rho }-\log
{t}}{L_{3} (t) }=-1,\quad \textrm{a.s.}
\end{equation}
\end{cor}

\begin{rem}\label{z2}
Relations \eqref{f28} and \eqref{f29} have been proved in \cite{do_m} by
direct calculations. Let us note that in case $m=\infty $, which has also
been treated in \cite{do_m}, the asymptotics of $\bar{Q}(t)$ is of completely
different form, see also \cite{mi_19}.
\end{rem}
\end{example}

\begin{example}[Birth and death processes]\label{example3}
Let $X=(X(t))_{t\geq 0}$ be a birth and death processes with parameters
%
\begin{equation}\label{f290}
\lambda _{n} = \lambda n +a, \qquad \mu _{n} = \mu n , \quad \lambda , \mu ,a
>0,\  n=0, 1, 2, \ldots ,
\end{equation}
see \cite[Ch. 7, $\S 6 $]{kar1}. That is, $(X(t))_{t\geq 0}$ is a
time-homogeneous Markov process such that, for $t\geq 0$, given $\{X(t)=n\}$
the probability of transition to state $n + 1$ over a small period of time
$\delta $ is $(\lambda n +a)\delta +o(\delta )$, $n=0,1,2,3,\ldots{} $, and
the probability of transition to $n - 1$ is $\mu n \delta +o(\delta )$,
$n=1,2,3,\ldots{} $. The parameter $a$ can be interpreted as the
infinitesimal intensity of population growth due to immigration. The
birth--death process $X$ is usually called the process with linear growth and
immigration.

We assume that $X(0)=0$ and
%
\begin{equation}\label{f32}
\rho := \frac{\lambda }{ \mu } <1 .
\end{equation}
Put
\begin{equation*}
\theta _{0} = 1, \qquad \theta _{k} = \prod _{i=1}^{k} \frac{\lambda
_{i-1}}{\mu _{i}} , \quad k\in \mathbb{N} .
\end{equation*}
It is not difficult to check that condition \eqref{f32} ensures
%
\begin{equation}\label{f33}
\sum _{k\geq 1} \theta _{k} < \infty ,
\end{equation}
and
%
\begin{equation}\label{f34}
\sum _{k\geq 1} \frac{1}{\lambda _{k} \theta _{k}} = \infty .
\end{equation}
Under conditions \eqref{f33} and \eqref{f34}, see \cite{kar1} and
\cite{kar2}, there exists a stationary regime, that is,
\begin{equation*}
\lim _{t \rightarrow \infty }\mathbf{P}( X(t)= k) = p_{k} ,
\end{equation*}
with
%
\begin{equation}\label{f134}
p_{k} = \theta _{k} p_{0},\quad k=0,1,2,3,\ldots , \quad \text{where } p_{0}
= \Biggl(\sum _{k= 0}^{\infty } \theta _{k} \Biggr)^{-1}.
\end{equation}

Further, $X$ is a regenerative process\index{regenerative processes} with regeneration moments $(S_{k})$,
where\break $S_{0} =0$ and $S_{i}$, $i\in \mathbb{N}$, is the moment of $i$-th
return to state $0$. It is known that
\begin{equation*}
\mathbf{E}T_{k} = \frac{1 }{(\lambda _{0}+\mu _{0})p_{0} }= \frac{1 }{a p_{0}
} ,
\end{equation*}
where $T_{k} =S_{k} - S_{k-1}$ is the duration of the $k$-th regeneration
cycle, see Eq. (32) in \cite{ok_ik}. If \eqref{f32} holds, then
\begin{equation*}
M(t) := \mathbf{E} X(t)\rightarrow \frac{a}{\mu - \lambda },\quad t \to
\infty ,
\end{equation*}
see \cite{kar1}. We are interested in the asymptotic behavior of extreme
values
\begin{equation*}
\bar{X}(t) = \sup _{0\leq s < t} X(s), \quad t\geq 0.
\end{equation*}
Let us show how to apply Theorem~\ref{thm:main2} to the asymptotic analysis
of $\bar{X}(t)$. Firstly, we need to evaluate accurately the sequence
$(R(n))$ defined by
\begin{equation*}
q(n): = \mathbf{P} ( \bar{X}(T_{1}) > n ) = \exp (-R(n)).
\end{equation*}
It is known, see \cite{asm} or Eq. (34) in \cite{ok_ik}, that
%
\begin{equation}\label{f35}
q(n) = \frac{1}{ \sum _{k=0}^{n} \alpha _{k} },
\end{equation}
where $\alpha _{0} = 1$ and $\alpha _{k} = \prod _{i=1}^{k} \frac{\mu
_{i}}{\lambda _{i}}$ for $k\in \mathbb{N}$.

Using equalities \eqref{f290} and \eqref{f32} we can rewrite $\alpha _{k}$ as
follows:
%
\begin{equation}\label{f36}
\alpha _{k} = \frac{\beta _{k}}{\rho ^{k} } , \qquad \beta _{k} = \prod
_{i=1}^{k} \left (1 - \frac{1}{1+i\lambda /a}\right ).
\end{equation}
Further, using Taylor's expansion
\begin{equation*}
\log (1 + x) = x - \frac{x^{2}}{2} + \frac{x^{3}}{3} - \cdots , \quad |x| <
1,
\end{equation*}
we have
%
\begin{equation}\label{f37}
\log \beta _{k} = \sum _{i=1}^{k} \log \left (1 - \frac{1}{1+i\lambda
/a}\right ) = - \sum _{i=1}^{k} \frac{1}{1+i\lambda /a} + d_{k} ,
\end{equation}
where $d_{k}$ has a finite limit, as $k\to \infty $. Combining the relation
%
\begin{equation}\label{f39}
\left |\sum _{i=1}^{k} \frac{1}{1+i\lambda /a} - \sum _{i=1}^{k}
\frac{1}{i\lambda /a} \right | = \sum _{i=1}^{k} \frac{a}{\lambda i
(1+i\lambda /a)} =C_{1} +o(1),\quad k\to \infty ,
\end{equation}
and the fact
%
\begin{equation}\label{f38}
\lim _{n\rightarrow \infty } \sum _{i=1}^{n} \frac{1}{i} - \log n = \gamma ,
\end{equation}
with $\gamma =0.577\ldots $ being the Euler--Mascheroni constant, we conclude
\begin{equation*}
\sum _{i=1}^{k} \log \left (1 - \frac{1}{1+i\lambda /a}\right ) = -
\frac{a}{\lambda } \log k + C_{2} + o(1),\quad k\to \infty .
\end{equation*}
Therefore,
%
\begin{equation}\label{f40}
\beta _{k} =C k^{-a/\lambda } (1+o(1)) ,
\end{equation}
where
%
\begin{equation}\label{f140}
C=e^{C_{2}}:=\lim _{n\rightarrow \infty } n^{a/\lambda } \prod _{i=1}^{n} \biggl(1
- \frac{1}{1+i\lambda /a}\biggr).
\end{equation}
Now we can apply Lemma~\ref{l5} to obtain
\begin{equation*}
\Lambda _{n} = \sum _{k=1}^{n} \frac{\rho ^{-k}}{k^{a/\lambda }} = \frac{\rho
^{-n-1}}{(\frac{1}{\rho }-1)n^{a/\lambda }} (1+o(1)), \quad n\to \infty .
\end{equation*}
Taking into account equality \eqref{f40}, we obtain
\begin{equation*}
\sum _{k=0}^{n} \alpha _{k} = \frac{C \rho ^{-n-1}}{(1/\rho -1)n^{a/\lambda
}} (1+o(1)),\quad n\to \infty .
\end{equation*}
Thus,
%
\begin{equation}\label{f53}
q(n) = \frac{ 1/\rho -1}{C } \rho ^{n+1} n^{a/\lambda }(1+o(1)), \quad n\to
\infty ,
\end{equation}
and we have the following representation
\begin{equation*}
R(n)= -\log q(n) = R_{0} (n) + R_{1} (n),
\end{equation*}
where
\begin{equation*}
R_{0} (n) = -n \log \rho - \frac{a}{\lambda } \log n , \qquad R_{1} (n)
= -\log \frac{ 1/\rho -1}{C }-\log \rho + o(1),\quad n\to \infty .
\end{equation*}
The function $ R_{0} (x) = -x \log \rho - \frac{a}{\lambda } \log x $ is
increasing for $x \geq x_{0} = -\frac{a}{\lambda \log \rho }$ and,
furthermore, satisfies condition $(\mathbb{U})$.

Since $r_{0} (x) = -\log \rho - \frac{a}{\lambda x}$, conditions
\eqref{f154}, \eqref{f155} of Theorem~\ref{thm:main2} hold.

It remains to find a simple formula for the function $A_{0}$ in
\eqref{f200} and \eqref{f300}. To this end, let us write
\begin{equation*}
\log \left (\frac{t}{\alpha _{T}}\right )= R_{0}(A_{0} (t)) = -A_{0} (t)
\log \rho - \frac{a}{\lambda } \log A_{0} (t)=A_{0}(t)\left (-\log
\rho +o(1)\right ),
\end{equation*}
as $t\to \infty $. Upon taking logarithms we get
\begin{equation*}
\log A_{0} (t) = \log \log \left (\frac{t}{\alpha _{T}}\right ) +O(1),
\quad t\to \infty ,
\end{equation*}
and plugging this back into the initial equality yields
\begin{equation*}
A_{0} (t) =
\frac{\log {t} + \frac{a}{\lambda }L_{2} (t) + O(1)}{-\log \rho }.
\end{equation*}
Thus, from Theorem~\ref{thm:main2} we infer the following.
%
\begin{cor}\label{n3}
Let $(X(t))_{t\geq 0}$ be the birth and death process with parameters
$\lambda _{n} = \lambda n +a$, $\mu _{n} = \mu n$, where $\lambda ,\mu ,
a>0$, $n=0, 1, 2,3,\ldots{}$. Suppose also that \eqref{f32} holds. Then
%
\begin{equation}
\label{f51}
\limsup _{t\rightarrow \infty }
\frac{\bar{X}(t)\log \frac{1}{\rho }-\log t}{L_{2} (t) }=1+
\frac{a}{\lambda },\quad \textrm{a.s.},
\end{equation}
and
%
\begin{equation}
\label{f52}
\liminf _{t\rightarrow \infty }
\frac{\bar{X}(t)\log \frac{1}{\rho }-\log t-\frac{a}{\lambda }L_{2} (t)}{L_{3} (t) }=-1,
\quad \textrm{a.s.}
\end{equation}
\end{cor}

\begin{rem}%
\label{z3}
It is clear from the above calculations that the condition $ X(0)=0$ can
be removed, that is, equations \eqref{f51} and \eqref{f52} hold true for
an arbitrary starting point $X(0)\in \{0,1,2,\ldots \}$.
\end{rem}

Let us finally mention without a proof a statement which follows easily
from equations \eqref{f134}, \eqref{f53} and Theorem 2 in
\cite{ok_ik}.

\begin{cor}%
\label{n4}
Let $(X(t))_{t\geq 0}$ be the birth and death process that satisfies all
conditions of Corollary \ref{n3}. Then
%
\begin{equation}
\lim _{n\rightarrow \infty }\mathbf{P}(\bar{X}(t_{n}^{*})\geq n) = 1-
\exp (-a p_{0} x),\quad x>0,
\end{equation}
where $t^{*}_{n} = C \rho ^{-n} n^{-a/ \lambda }x/(1/\rho -1)$, while
$p_{0}$ and $C$ are defined by \eqref{f134} and \eqref{f140}. This relation
can also be recast in a more transparent way as follows:
%
\begin{equation}
\lim _{n\to \infty }\mathbf{P}(C^{-1}(1/\rho -1)\rho ^{n} n^{a/
\lambda }X^{-1}(n)> x)=\exp (-a p_{0} x),
\end{equation}
where $X^{-1}(n)=\inf \{t\in \mathbb{R}:X(t)=n\}$ is the first time when
$(X(t))_{t\geq 0}$ visits state $n\in \mathbb{N}$, that is, the distribution
of $\rho ^{n} n^{a/\lambda }X^{-1}(n)$ converges to an exponential law.
\end{cor}
\end{example}


\begin{acknowledgement}[title={Acknowledgments}]
We thank anonymous referees for a number of useful suggestions and remarks.
\end{acknowledgement}

\bibliographystyle{vmsta-mathphys}


\end{document}